\documentclass[]{elsarticle}
\usepackage{amsmath,amsfonts,amsthm}
\usepackage{pifont}
\usepackage{graphicx}
\usepackage{hyperref}
\usepackage{tikz}
\usepackage{showlabels}

\newtheorem{thm}{Theorem}[section]

\newtheorem{defn}[thm]{Definition}

\newtheorem{exa}{Example}[section]

\newcommand{\si}{\sigma}

\newcommand{\dif}{\mathrm{d}}

\newcommand{\ba}{\begin{array}}
\newcommand{\ea}{\end{array}}

\begin{document}
\begin{frontmatter}
\numberwithin{equation}{section}
\title{Symplecticity-preserving continuous-stage Runge-Kutta-Nystr\"{o}m methods}
\author[a]{Wensheng Tang}
\ead{tangws@lsec.cc.ac.cn}
\address[a]{School of Mathematics and Statistics,\\
Changsha University of Science and Technology,\\ Changsha $410114$,
China}
\author[b]{Jingjing Zhang\corref{cor1}}
\ead{zhangjj@hpu.edu.cn}\cortext[cor1]{Corresponding author.}
\address[b]{School of Mathematics and Information Science,
\\ Henan Polytechnic University, Jiaozuo $454001$, China}
\author[]{}


\begin{abstract}

We develop continuous-stage Runge-Kutta-Nystr\"{o}m (csRKN) methods
for solving second order ordinary differential equations (ODEs) in
this paper. The second order ODEs are commonly encountered in
various fields and some of them can be reduced to the first order
ODEs with the form of separable Hamiltonian systems. The
symplecticity-preserving numerical algorithm is of interest for
solving such special systems. We present a sufficient condition for
a csRKN method to be symplecticity-preserving, and by using Legendre
polynomial expansion we show a simple way to construct such
symplectic RKN type method.

\end{abstract}

\begin{keyword}
Hamiltonian systems; Symplecticity-preserving; Continuous-stage
Runge-Kutta-Nystr\"{o}m methods; Legendre polynomial; Symplectic
conditions.


\end{keyword}

\end{frontmatter}

\section{Introduction}

It is well-known that Runge-Kutta (RK) methods, partitioned
Runge-Kutta (PRK) methods and Runge-Kutta-Nystr\"{o}m (RKN) methods
paly a central role in the context of numerical solution of ordinary
differential equations (ODEs), and they were well-developed in the
previous investigations
\cite{butcher87tna,hairernw93sod,hairerw96sod}.

More recently, numerical methods with infinitely many stages
including continuous-stage Runge-Kutta (csRK) method,
continuous-stage partitioned Runge-Kutta (csPRK) method have been
investigated and discussed in
\cite{hairer10epv,Tangs12ana,Tangs12tfe,Tangs14cor,miyatake14aee,
Tanglx15cos,butcherm15aco}. It is found that based on such methods
we can obtain many classical RK methods and PRK methods of
arbitrarily high order by using quadrature formulae but without
resort to solving the tedious nonlinear algebraic equations that
stem from the order conditions with many unknown coefficients. The
construction of continuous-stage numerical methods seems more easier
than that of those classical methods, since the associated Butcher
tableau coefficients belong to the space of continuous functions and
they can be treated in use of orthogonal polynomial expansions
\cite{Tangs14cor,Tanglx15cos}. Moreover, as shown in
\cite{Tangs14cor,Tanglx15cos}, numerical methods serving some
special purpose including symplecticity-preserving methods for
Hamiltonian systems, symmetric methods for reversible systems,
energy-preserving methods for Hamiltonian systems, numerical methods
with conjugate symplecticity (up to a finite order) for Hamiltonian
systems can also be constructed and investigated based on such new
framework.

It is worth mentioning that some methods with special purpose
couldn't possibly exist in the classical context of numerical
methods but it does within the new framework. For instance,
\cite{Celledoni09mmoqw} has proved that there is no
energy-preserving RK methods for general Hamiltonian system
excluding those polynomial system, but energy-preserving methods
based on csRK obviously exist
\cite{hairer10epv,quispelm08anc,brugnanoit10hbv,Tangs12ana,
Tangs12tfe,Tangs14cor,miyatake14aee,butcherm15aco}. It is also found
that some Galerkin variational methods can be related to
continuous-stage (P)RK methods, which can not be completely
explained in the classical (P)RK framework
\cite{Tangs12tfe,Tangs15sdg,Tangsc15dgm}. As a consequence, the
continuous-stage methods provide a new broader scope for numerical
solution of ODEs and they are worth further investigating.

As is well known, the second order ODEs are commonly encountered in
various fields including celestial mechanics, molecular dynamics,
biological chemistry and so on
\cite{hairernw93sod,sanzc94nhp,hairerlw06gni}. In this paper, we are
going to develop continuous-stage RKN (csRKN) methods for solving
second order ODEs. In particular, there is a number of second order
ODEs that can be reduced to the first order ODEs with the form of
separable Hamiltonian systems, and the symplecticity-preserving
discretization for such systems is of considerable interest
\cite{Feng84ods,sanzc94nhp,hairerlw06gni}. For this sake, we will
present a sufficient condition for a csRKN method to be
symplecticity-preserving, and then show the construction of
symplectic RKN type methods by using the Legendre polynomial
expansion technique.

The outline of this paper is as follows. In the next section, we
introduce the so-called csRKN methods for solving second order ODEs.
After that we present the corresponding symplectic conditions and
the order conditions, then we use the orthogonal polynomial
expansion technique to construct symplecticity-preserving csRKN
methods, which will be given in section 3-4. Section 5 is devoted to
discuss the construction of diagonally implicit symplectic methods.
At last, the concluding remarks will be given.

\section{Continuous-stage RKN method}

In the field of engineering and physics there are a large class of
problems which can be expressed by a system of second order
differential equations
\begin{equation}\label{eq:second}
\ddot{q}=f(t, q),\;q\in \mathbb{R}^d,
\end{equation}
where the double dots on $q$ represent the second-order derivative
with respect to $t$ and
$f:\mathbb{R}\times\mathbb{R}^{d}\rightarrow\mathbb{R}^{d}$ is a
sufficiently smooth vector function.

For system \eqref{eq:second}, the often used treatment is to write
it as a first order differential system by introducing $p=\dot{q}$,
namely
\begin{equation}\label{eq:first}
\begin{cases}
\dot{q}=p,\\[2pt]
\dot{p}=f(t, q).
\end{cases}
\end{equation}

As presented in \cite{Tanglx15cos}, by using a continuous-stage
partitioned Runge-Kutta (csPRK) method to \eqref{eq:first}, it gives
\begin{subequations}
\begin{alignat}{2}
\label{eq:cs1}
&Q_\tau = q_n +h\int_{0}^{1} A_{\tau,\si} P_{\si} d\si, \quad\tau \in[0, 1], \\
\label{eq:cs2}
&P_\tau = p_n +h\int_{0}^{1} \hat{A}_{\tau,\si}
f(t_n+C_\si h, Q_\si) d \si, \quad\tau \in[0, 1], \\
\label{eq:cs3}
&q_{n+1} = q_n +h\int_{0}^{1} B_\tau P_\tau d\tau, \quad n\in \mathbb{N},\\
\label{eq:cs4} &p_{n+1} = p_n +h\int_{0}^{1} \hat{B}_\tau
f(t_n+C_\tau h, Q_\tau) d\tau,\quad n\in \mathbb{N},
\end{alignat}
\end{subequations}
where $A_{\tau, \si},\,\hat{A}_{\tau, \si}$ are functions of two
variables $\tau, \si\in[0,1]$ and $B_\tau,\;\hat{B}_\tau,\;C_\tau$
are functions of $\tau\in[0,1]$. We call $Q_\tau$ and $P_\tau$ the
internal continuous stages. In addition, here we assume that
$\int_{0}^{1} A_{\tau,\si} d\si=\int_{0}^{1} \hat{A}_{\tau,\si}
d\si=C_\tau$, and $\int_{0}^{1}B_\tau d\tau=\int_{0}^{1}\hat{B}_\tau
d\tau=1$.

By inserting \eqref{eq:cs2} into \eqref{eq:cs1}, we derive
\begin{align}
Q_\tau& = q_n +h\int_{0}^{1} A_{\tau,\si}\big( p_n +h\int_{0}^{1}
\hat{A}_{\si, \rho} f(t_n+C_\rho h, Q_\rho) d \rho\big)  d\si \\
&=q_n +hC_\tau p_n +h^2\int_{0}^{1} \bar{A}_{\tau, \rho}
f(t_n+C_\rho h, Q_\rho) d \rho,
 \end{align}
where we define $\bar{A}_{\tau,
\rho}=\int_{0}^{1}A_{\tau,\si}\hat{A}_{\si, \rho }d\si$ and here by
hypothesis $C_\tau=\int_{0}^{1} A_{\tau,\si} d\si$. Similarly, by
inserting \eqref{eq:cs2} into \eqref{eq:cs3}, we have
\begin{alignat}{2}
q_{n+1}& = q_n +h\int_{0}^{1} B_\tau \big( p_n +h\int_{0}^{1}
\hat{A}_{\tau, \si} f(t_n+C_\si h, Q_\si) d \si\big) d\tau \\
&=q_n+ h p_n+h^2 \int_{0}^{1} \bar{B}_\si  f(t_n+C_\si h, Q_\si)
d\si
\end{alignat}
where we denote $ \bar{B}_\si=\int_{0}^{1}B_\tau \hat{A}_{\tau,
\si}d\tau$, and note that $\int_{0}^{1}B_\tau d\tau=1$.

In summary, by using a csPRK method to \eqref{eq:first} and
eliminating the internal stage variable $P_\tau$, we can obtain the
method which is referred to as a continuous-stage
Runge-Kutta-Nystr\"{o}m method in this paper.

\begin{defn}[Continuous-stage Runge-Kutta-Nystr\"{o}m method]
Let $\bar{A}_{\tau, \si}$ be a function of variables $\tau,
\si\in[0,1]$ and $\bar{B}_\tau,\;\hat{B}_\tau,\;C_\tau$ be functions
of $\tau\in[0,1]$.  A continuous-stage Runge-Kutta-Nystr\"{o}m
(csRKN) method  for the solution of \eqref{eq:second} is given by
\begin{subequations}
    \begin{alignat}{2}
    \label{eq:csrkn1}
&Q_\tau=q_n +hC_\tau p_n +h^2\int_{0}^{1} \bar{A}_{\tau, \si}
f(t_n+C_\si h, Q_\si) d \si, \;\;\tau \in[0, 1], \\
    \label{eq:csrkn2}
&q_{n+1}=q_n+ h p_n+h^2 \int_{0}^{1} \bar{B}_\tau f(t_n+C_\tau h, Q_\tau) d\tau, \quad n\in \mathbb{N},\\
    \label{eq:csrkn3}
&p_{n+1} = p_n +h\int_{0}^{1} \hat{B}_\tau f(t_n+C_\tau h, Q_\tau)
d\tau,\quad n\in \mathbb{N},
    \end{alignat}
\end{subequations}
which can be characterized by the following Butcher tableau\\
\centerline{
    \begin{tabular}{c|c}\\[-4pt]
        $C_\tau$ & $\bar{A}_{\tau,\si}$ \\[2pt]
        \hline\\[-10pt]
        $ $ &$\bar{B}_\tau$ \\
        \hline\\[-10pt]
            $ $ &$\hat{B}_\tau$ \\
    \end{tabular}
}
\end{defn}

\section{Symplectic conditions for csRKN method}

When the system \eqref{eq:second} is autonomous (i.e.,
time-independent for the right-hand-side vector field) and $f$ is
the gradient of a scalar function, e.g., $$f(q)=-\nabla_qV(q),$$
then it becomes a separable Hamiltonian system in the form
$$\dot{z}=J^{-1}\nabla_{z}\mathbf{H}(z),\;z=(p, q)\in \mathbb{R}^{2d},$$
with the Hamiltonian $\mathbf{H}(z)=\frac{1}{2}p^Tp+V(q)$ and
$J=\left(\begin{array}{cccc}
                  0 & I_{d\times d}\\
                  -I_{d\times d} & 0
          \end{array}\right)$, and such system possesses an intrinsic
geometric structure called symplecticity. This states that the flow
$\varphi_t$ of the system is a symplectic transformation
\cite{hairerlw06gni}, i.e., $(\varphi'_t)^TJ\varphi'_t=J,$ where
$\varphi'_t$ denotes the derivative of $\varphi_t$ with respect to
the initial values. For Hamiltonian system, symplectic numerical
method is of considerable interest \cite{Feng84ods,hairerlw06gni},
as it always exhibits bounded small energy errors for the
exponentially long time, and can correctly mimic the qualitative
behavior of the original system (e.g., preserving the quasi-periodic
orbits (namely KAM tori) and chaotic regions of phase space
\cite{shang99kam}). A one-step method $\Phi_h:
(p_n,\,q_n)\mapsto(p_{n+1},\,q_{n+1})$, when applied to a
Hamiltonian system, is called symplectic if and only if
$(\Phi'_h)^TJ\Phi'_h=J$, or equivalently $dp_{n+1}\wedge
dq_{n+1}=dp_n \wedge dq_n$ (Here $\wedge$ denotes the standard wedge
product in differential geometry).

\subsection{The sufficient condition for symplecticity}

\begin{thm}\label{conditions: sym}
If a csRKN method denoted by
$(\bar{A}_{\tau,\si},\bar{B}_\tau,\hat{B}_\tau,C_\tau)$ satisfies
\begin{subequations}
\begin{alignat}{2}
\label{sym_cond_orig01}
&\hat{B}_\tau(1-C_\tau)=\bar{B}_\tau,\;\tau\in[0,1],\\
\label{sym_cond_orig02}
&\hat{B}_\tau(\bar{B}_\si-\bar{A}_{\tau,\si})=\hat{B}_\si(\bar{B}_\tau
-\bar{A}_{\si,\tau}),\;\tau,\si\in[0,1],
\end{alignat}
\end{subequations}
then the method is symplectic for solving the autonomous second
order differential equations \eqref{eq:second} with
$f(q)=-\nabla_qV(q)$ (which can be rewritten as a Hamiltonian
system).
\end{thm}
\begin{proof} By the csRKN method (\ref{eq:csrkn1}-\ref{eq:csrkn3}), we have
    \begin{equation}\label{eq:proofsym1}
        \begin{split}
            &\; dp_{n+1} \wedge dq_{n+1}\\
            &=d( p_n +h\int_{0}^{1} \hat{B}_\tau f(Q_\tau) d\tau) \wedge
            d(q_n+ h p_n+h^2 \int_{0}^{1} \bar{B}_\tau  f(Q_\tau) d\tau) \\
            &=dp_n \wedge dq_n + h \int_{0}^{1} (\hat{B}_\tau d f(Q_\tau)\wedge dq_n )d\tau +h dp_n \wedge dp_n \\
            &\quad+h^2 \int_{0}^{1} (\hat{B}_\tau d f(Q_\tau)\wedge dp_n) d\tau
            +h^2 \int_{0}^{1} (\bar{B}_\tau dp_n \wedge d f(Q_\tau)) d\tau \\ &
            \quad+h^3\int_{0}^{1}\int_{0}^{1} \hat{B}_\tau\bar{B}_\si  d f(Q_\tau) \wedge d f(Q_\si)d\si d \tau
        \end{split}
    \end{equation}
Because of the skew symmetry of wedge product, the third term
vanishes. By virtue of \eqref{eq:csrkn1}, the second term can be
recast as
\begin{equation}\label{eq:sovqn}
    \begin{split}
            &\;h \int_{0}^{1} (\hat{B}_\tau d f(Q_\tau) \wedge dq_n )d\tau\\
            &=h \int_{0}^{1} \big( \hat{B}_\tau d f(Q_\tau) \wedge
            d(Q_\tau-hC_\tau p_n -h^2\int_{0}^{1} \bar{A}_{\tau, \si} f(Q_\si) d \si)\big)d\tau \\
            &=h \int_{0}^{1}(\hat{B}_\tau d f(Q_\tau) \wedge dQ_\tau)
            d\tau- h^2 \int_{0}^{1} (\hat{B}_\tau C_\tau d f(Q_\tau) \wedge d p_n )d\tau   \\
            &\quad-h^3 \int_{0}^{1} (\int_{0}^{1} \hat{B}_\tau
            \bar{A}_{\tau, \si} d f(Q_\tau) \wedge d f(Q_\si) d \si ) d\tau   \\
    \end{split}
\end{equation}
Note that  $f(q)=-\nabla_q V(q)$ and its Jacobian matrix is
symmetric, the first term in the above equality vanishes. Then by
substituting \eqref{eq:sovqn} into \eqref{eq:proofsym1}, it yields
    \begin{equation}\label{eq:proofsym2}
    \begin{split}
    &\; dp_{n+1} \wedge dq_{n+1}\\
    &=dp_n \wedge dq_n- h^2 \int_{0}^{1}(\hat{B}_\tau C_\tau d f(Q_\tau) \wedge d p_n)d\tau\\
    &\quad-h^3 \int_{0}^{1} \int_{0}^{1}(\hat{B}_\tau \bar{A}_{\tau, \si} d f(Q_\tau)
    \wedge d f(Q_\si)) d \si d\tau +h^2 \int_{0}^{1} (\hat{B}_\tau d f(Q_\tau)\wedge dp_n) d\tau\\
    &\quad-h^2 \int_{0}^{1} (\bar{B}_\tau d f(Q_\tau) \wedge dp_n ) d\tau
    \;+h^3\int_{0}^{1}\int_{0}^{1} \hat{B}_\tau\bar{B}_\si  d f(Q_\tau) \wedge d f(Q_\si)d\si d \tau \\
    &=dp_n \wedge dq_n+h^2 \int_{0}^{1}(-\hat{B}_\tau C_\tau+\hat{B}_\tau-\bar{B}_\tau )d f(Q_\tau)\wedge dp_n d\tau  \\
    &\quad+h^3\int_{0}^{1} \int_{0}^{1} (\hat{B}_\tau\bar{B}_\si-\hat{B}_\tau \bar{A}_{\tau, \si} )d f(Q_\tau) \wedge d f(Q_\si) d \si d\tau
    \end{split}
    \end{equation}
The last term of the formula above can be reshaped as follows
\begin{equation}
\begin{split}
&\; h^3\int_{0}^{1} \int_{0}^{1} (\hat{B}_\tau\bar{B}_\si-\hat{B}_\tau \bar{A}_{\tau, \si} )d f(Q_\tau) \wedge d f(Q_\si) d \si d\tau\\
&=-\frac{h^3}{2}\int_{0}^{1} \int_{0}^{1} \big(\hat{B}_\tau
\bar{A}_{\tau, \si} d f(Q_\tau) \wedge d f(Q_\si)+\hat{B}_\si
\bar{A}_{\si, \tau} d f(Q_\si) \wedge d f(Q_\tau)\big) d \si d\tau\\
&\quad+\frac{h^3}{2}\int_{0}^{1}\int_{0}^{1}
\big(\hat{B}_\tau\bar{B}_\si d f(Q_\tau) \wedge d
f(Q_\si)+\hat{B}_\si\bar{B}_\tau df(Q_\si)
\wedge d f(Q_\tau)\big)d\si d \tau \\
&=-\frac{h^3}{2}\int_{0}^{1} \int_{0}^{1} \big(\hat{B}_\tau
\bar{A}_{\tau, \si} d f(Q_\tau) \wedge d f(Q_\si)-\hat{B}_\si
\bar{A}_{\si, \tau} d f(Q_\tau) \wedge d f(Q_\si)\big) d \si d\tau\\
&\quad+\frac{h^3}{2}\int_{0}^{1}\int_{0}^{1}
\big(\hat{B}_\tau\bar{B}_\si d f(Q_\tau) \wedge d
f(Q_\si)-\hat{B}_\si\bar{B}_\tau d f(Q_\tau)
\wedge d f(Q_\si)\big)d\si d \tau \\
&=\frac{h^3}{2} \int_{0}^{1}\int_{0}^{1}(\hat{B}_\tau\bar{B}_\si
-\hat{B}_\si\bar{B}_\tau -\hat{B}_\tau \bar{A}_{\tau,\si}+
\hat{B}_\si \bar{A}_{\si,\tau})df(Q_\tau) \wedge df(Q_\si) d \si
d\tau
\end{split}
\end{equation}
Therefore, if we require the condition given by
(\ref{sym_cond_orig01}-\ref{sym_cond_orig02}), then the last two
terms in \eqref{eq:proofsym2} vanish, and it yields $$dp_{n+1}\wedge
dq_{n+1}=dp_n \wedge dq_n,$$ which implies the symplecticity.
\end{proof}

A very special class of separable Hamiltonian systems commonly
considered in practice is the system with Hamiltonian
$$\mathbf{H}(p,q)=\frac{1}{2}p^TM^{-1}p+V(q),$$ where $M$ (mass matrix) is a
constant, symmetric and invertible matrix, and the corresponding
Hamiltonian system becomes
\begin{equation*}
\begin{cases}
\dot{q}=M^{-1}p,\\[2pt]
\dot{p}=-\nabla_q V(q).
\end{cases}
\end{equation*}
If we let $\widetilde{p}=M^{-1}p$, then we get
\begin{equation*}
\begin{cases}
\dot{q}=\widetilde{p},\\[2pt]
\dot{\widetilde{p}}=-M^{-1}\nabla_q V(q),
\end{cases}
\end{equation*}
which is in the form \eqref{eq:first}. By eliminating
$\widetilde{p}$, it reads
\begin{equation}\label{eq:Hs}
\ddot{q}=-M^{-1}\nabla_q V(q).
\end{equation}
For such a second order system, the csRKN method is
\begin{equation}\label{scheme1}
\begin{split}
&Q_\tau=q_n +hC_\tau \widetilde{p}_n +h^2\int_{0}^{1} \bar{A}_{\tau, \si} f(Q_\si) d \si, \;\;\tau \in[0, 1], \\
&q_{1}=q_n+ h \widetilde{p}_n+h^2 \int_{0}^{1} \bar{B}_\tau  f(Q_\tau) d\tau, \\
&\widetilde{p}_{n+1}= \widetilde{p}_n +h\int_{0}^{1} \hat{B}_\tau
f(Q_\tau) d\tau,
\end{split}
\end{equation}
where $f(q)=-M^{-1}\nabla_q V(q)$ and $\widetilde{p}_n=M^{-1}p_n$.

By Theorem \ref{conditions: sym}, if we require
(\ref{sym_cond_orig01}-\ref{sym_cond_orig02}), then the one-step
method \eqref{scheme1} mapping $(\widetilde{p}_n, q_n)$ to
$(\widetilde{p}_{n+1}, q_{n+1})$ is symplectic, i.e.,
$$d\widetilde{p}_{n+1}\wedge dq_{n+1}=d\widetilde{p}_n \wedge
dq_n.$$ However, we are interested in the method in terms of the
variables $p$ and $q$, rather than in terms of
$\widetilde{p}=\dot{q}$ and $q$. To address this issue, we observe
that \eqref{scheme1} can be recast as
\begin{equation}\label{scheme2}
\begin{split}
&Q_\tau=q_n +hC_\tau M^{-1}p_n +h^2\int_{0}^{1} \bar{A}_{\tau, \si} f(Q_\si) d \si, \;\;\tau \in[0, 1], \\
&q_{1}=q_n+ h M^{-1}p_n+h^2 \int_{0}^{1} \bar{B}_\tau  f(Q_\tau) d\tau, \\
&p_{n+1}=p_n +hM\int_{0}^{1} \hat{B}_\tau f(Q_\tau) d\tau,
\end{split}
\end{equation}
where the last formula is derived by multiplying $M$ from the
left-hand side of \eqref{scheme1}. By the similar arguments as the
proof of Theorem \ref{conditions: sym}, we can prove that it still
yields\footnote{A detailed proof will be found in our another coming
paper.}
$$dp_{n+1}\wedge dq_{n+1}=dp_n \wedge dq_n.$$
Therefore, the csRKN method \eqref{scheme2} remains symplectic under
the conditions (\ref{sym_cond_orig01}-\ref{sym_cond_orig02}).

\subsection{Further characterizations for symplecticity}

In what follows, we will show another useful result which shows the
characterizations for a csRKN method to be symplectic.

Now we introduce the $\iota$-degree normalized shifted Legendre
polynomial $P_\iota(x)$ by using the Rodrigues formula
$$P_0(x)=1,\;P_\iota(x)=\frac{\sqrt{2\iota+1}}{\iota!}\frac{{\dif}^\iota}{\dif x^\iota}
\Big(x^\iota(x-1)^\iota\Big),\; \;\iota=1,2,3,\cdots.$$ A well-known
property of such Legendre polynomials is that they are orthogonal to
each other with respect to the $L^2$ inner product in $[0,\,1]$
$$\int_0^1 P_\iota(t) P_\kappa(t)\,\dif t= \delta_{\iota\kappa},\quad\iota,\,
\kappa=0,1,2,\cdots,$$ and they as well satisfy the following
integration formulae
\begin{equation}\label{property}
\begin{split} &\int_0^xP_0(t)\,\dif
t=\xi_1P_1(x)+\frac{1}{2}P_0(x), \\
&\int_0^xP_\iota(t)\,\dif
t=\xi_{\iota+1}P_{\iota+1}(x)-\xi_{\iota}P_{\iota-1}(x),\quad
\iota=1,2,3,\cdots, \\
&\int_x^1P_\iota(t)\,\dif
t=\delta_{\iota0}-\int_0^{x}P_\iota(t)\,\dif t,\quad
\iota=0,1,2,\cdots,
\end{split}
\end{equation}
where $\xi_\iota=\frac{1}{2\sqrt{4\iota^2-1}}$ and $\delta_{ij}$ is
the Kronecker symbol.

Similarly as the continuous-stage (P)RK methods discussed in
\cite{Tanglx15cos}, we will use the simplifying hypothesis
$\hat{B}_\tau=1,C_\tau=\tau$ throughout this paper. By exploiting
the orthogonal polynomial expansions we get the following result.

\begin{thm}\label{construct_scsRKN}
The csRKN method denoted by
$(\bar{A}_{\tau,\si},\bar{B}_\tau,\hat{B}_\tau,C_\tau)$ with
$\hat{B}_\tau=1, C_\tau=\tau$ is symplectic for solving system
\eqref{eq:second} (as a separable Hamiltonian system), if
$\bar{A}_{\tau,\si}$ and $\bar{B}_\tau$ take the following forms in
terms of Legendre polynomials
\begin{equation}\label{sym_cond}
\begin{split}
&\qquad\qquad\qquad\bar{B}_\tau=1-\tau=\frac{1}{2}P_0(\tau)-\xi_1P_1(\tau),\quad\tau\in[0,1],\\
&\bar{A}_{\tau,\si}=\alpha_{(0,0)}+\alpha_{(0,1)}P_1(\si)+
\alpha_{(1,0)}P_1(\tau)+\sum\limits_{i+j>1}\alpha_{(i,j)}
P_i(\tau)P_j(\sigma),\;\;\tau,\si\in[0,1],
\end{split}
\end{equation}
where the expansion coefficients $\alpha_{(i,j)}$ as real parameters
satisfy
$$\alpha_{(0,0)}\in\mathbb{R},\;\,\alpha_{(0,1)}-\alpha_{(1,0)}
=-\frac{\sqrt{3}}{6},\;\,\alpha_{(i,j)}=\alpha_{(j,i)},\,i+j>1.$$
\end{thm}
\begin{proof}
By the assumption $\hat{B}_\tau=1, C_\tau=\tau$ and using
\eqref{sym_cond_orig01} we get
$$\bar{B}_\tau=1-\tau=\frac{1}{2}P_0(\tau)-\xi_1P_1(\tau),$$
inserting it into \eqref{sym_cond_orig02}, then it ends up with
\begin{equation}\label{eq:AB}
\bar{A}_{\tau,\,\sigma}-\bar{A}_{\sigma,\,\tau}=\tau-\sigma=
\xi_1(P_1(\tau)-P_1(\sigma))=\frac{\sqrt{3}}{6}(P_1(\tau)-P_1(\sigma)),
\end{equation}
in which we have used the known equality
$\tau=\frac{1}{2}P_0(\tau)+\xi_1P_1(\tau)$.

Next, assume $\bar{A}_{\tau,\,\sigma}$ can be expanded as a series
in terms of the orthogonal basis
$\left\{P_i(\tau)P_j(\sigma)\right\}_{i,j=0}^\infty$ in
$[0,1]\times[0,1]$, written in the form
\begin{equation*}
\bar{A}_{\tau,\,\sigma}=\sum\limits_{0\leq
i,j\in\mathbb{Z}}\alpha_{(i,j)}
P_i(\tau)P_j(\sigma),\quad\alpha_{(i,j)}\in \mathbb{R},
\end{equation*}
and then by exchanging $\tau\leftrightarrow\sigma$ we have
\begin{equation*}
\bar{A}_{\sigma,\,\tau}=\sum\limits_{0\leq
i,j\in\mathbb{Z}}\alpha_{(i,j)} P_i(\sigma)P_j(\tau)
=\sum\limits_{0\leq i,j\in\mathbb{Z}}\alpha_{(j,i)}
P_j(\sigma)P_i(\tau).
\end{equation*}
Substituting the above two expressions into \eqref{eq:AB} and
collecting the like basis, it gives
$$\alpha_{(0,0)}\in\mathbb{R},\;\,\alpha_{(0,1)}-\alpha_{(1,0)}
=-\frac{\sqrt{3}}{6},\;\,\alpha_{(i,j)}=\alpha_{(j,i)},\,i+j>1,$$
which completes the proof.
\end{proof}

\section{Symplectic RKN methods based on csRKN}

In this section, we discuss the construction of symplectic RKN
methods based on csRKN.

It is almost mandatory that the practical implementation of the
csRKN method \eqref{eq:csrkn1}-\eqref{eq:csrkn3} needs the use of
numerical quadrature formula. By applying the quadrature formula
$(b_i, c_i)_{i=1}^r$ of order $p$ to
\eqref{eq:csrkn1}-\eqref{eq:csrkn3}, with abuse of notations
$Q_i=Q_{c_i}$, we derive an $r$-stage classical RKN method
\begin{subequations}
    \begin{alignat}{2}
    \label{eq:rkn1}
    &Q_i=q_n +hC_i p_n +h^2\sum\limits_{j=1}^{r} b_j \bar{A}_{ij} f(Q_j),\quad i=1,\cdots, r, \\
    \label{eq:rkn2}
    &q_{n+1}=q_n+ h p_n+h^2\sum\limits_{i=1}^{r} b_i\bar{B}_i  f(Q_i),\quad n\in \mathbb{N},  \\
    \label{eq:rkn3}
    &p_{n+1}= p_n +h\sum\limits_{i=1}^{r} b_i\hat{B}_i f(Q_i),\quad n\in \mathbb{N},
    \end{alignat}
\end{subequations}
where $\bar{A}_{ij}=\bar{A}_{c_i,c_j}, \bar{B}_i=\bar{B}_{c_i},
\hat{B}_i=\hat{B}_{c_i}, C_i=C_{c_i}$, which can be formulated by
the following Butcher tableau
\begin{equation}\label{tab:rknbutcher}
\begin{tabular}{c|c c c }\\[-4pt]
$C_1$ & $b_1\bar{A}_{11}$ & $\cdots$ &$b_r\bar{A}_{1r}$\\
$\vdots$ & $\vdots$ & & $\vdots$ \\
$ C_r$ & $b_1\bar{A}_{r1}$ &$\cdots$ &$b_r\bar{A}_{rr}$ \\[2pt]
\hline\\[-8pt]
$ $ & $b_1\bar{B}_{1}$ & $\cdots$ &$b_r\bar{B}_{r}$  \\[2pt]
\hline\\[-8pt]
$ $ & $b_1\hat{B}_{1}$ &$\cdots$ & $b_r\hat{B}_{r}$
\end{tabular}
\end{equation}

In particular, if we use the hypothesis
$\bar{B}_\tau=\hat{B}_\tau(1-C_\tau),\,\hat{B}_\tau=1,\,
C_\tau=\tau$ for $\tau\in[0,1]$, we then get an $r$-stage RKN method
with tableau
\begin{equation}\label{RKN:qua}
\begin{tabular}{c|c c c }\\[-4pt]
$c_1$ & $b_1\bar{A}_{11}$ & $\cdots$ &$b_r\bar{A}_{1r}$\\
$\vdots$ & $\vdots$ & & $\vdots$ \\
$ c_r$ & $b_1\bar{A}_{r1}$ &$\cdots$ &$b_r\bar{A}_{rr}$\\
\hline\\[-8pt]
$ $ & $\bar{b}_1$ & $\cdots$ &$\bar{b}_r$  \\
\hline\\[-8pt]
$ $ & $b_1$ &$\cdots$ & $b_r$
\end{tabular}
\end{equation}
where $\bar{b}_i=b_i(1-c_i),\; i=1,\cdots,r$.

The following result implies that we can construct symplectic RKN
method via symplectic csRKN method with the help of a quadrature
formula.
\begin{thm}
If the csRKN method denoted by
$(\bar{A}_{\tau,\si},\bar{B}_\tau,\hat{B}_\tau,C_\tau)$ satisfies
the symplectic condition
(\ref{sym_cond_orig01}-\ref{sym_cond_orig02}), then the associated
RKN method \eqref{tab:rknbutcher} derived by using a quadrature
formula $(b_i, c_i)_{i=1}^r$ is still symplectic.
\end{thm}
\begin{proof}
Recall that the sufficient condition for a classical RKN method
denoted by $(\bar{a}_{ij},\,\bar{b}_i,\,b_i,\,c_i)$ to be symplectic
is \cite{sanzc94nhp}
\begin{equation*}
\begin{split}
&\bar{b}_i=b_i(1-c_i),\quad i=1,\cdots,r,\\
&b_i(\bar{b}_j-\bar{a}_{ij})=b_j(\bar{b}_i-\bar{a}_{ji}),\quad i,
j=1,\cdots,r.
\end{split}
\end{equation*}
By (\ref{sym_cond_orig01}-\ref{sym_cond_orig02}), we have the
following equalities
\begin{equation*}
\begin{split}
&\bar{B}_i=\hat{B}_i(1-C_i),\quad i=1,\cdots,r,\\
&\hat{B}_i(\bar{B}_j-\bar{A}_{ij})=\hat{B}_j(\bar{B}_i-\bar{A}_{ji}),\quad
i, j=1,\cdots,r.
\end{split}
\end{equation*}
Therefore, the coefficients
$(b_j\bar{A}_{i,j},\,b_i\bar{B}_i,\,b_i\hat{B}_i,\,C_i)$ of the
associated RKN method satisfy
\begin{equation*}
\begin{split}
&b_i\bar{B}_i=b_i\hat{B}_i(1-C_i),\quad i=1,\cdots,r,\\
&b_i\hat{B}_i(b_j\bar{B}_j-b_j\bar{A}_{ij})=b_j\hat{B}_j(b_i\bar{B}_i-b_i\bar{A}_{ji}),\quad
i, j=1,\cdots,r,
\end{split}
\end{equation*}
which completes the proof by using the classical result.
\end{proof}

\subsection{Order conditions for RKN type methods}

To construct symplectic RKN method with a preassigned order, let us
introduce the order conditions for RKN type methods.

It is known that the classical RKN method for solving
\eqref{eq:second} can be formulated as
\begin{subequations}
    \begin{alignat}{4}
    \label{eq:grkn1}
    &Q_i=q_n +hc_i p_n +h^2\sum\limits_{j=1}^r \bar{a}_{i j} f(t_n+c_j h, Q_j), \quad i=1,\cdots, r, \\
    \label{eq:grkn2}
    &q_{n+1}=q_n+ h p_n+h^2\sum\limits_{i=1}^{r}\bar{b}_if(t_n+c_i h, Q_i), \quad n\in \mathbb{N},\\
    \label{eq:grkn3}
    &p_{n+1}= p_n +h\sum\limits_{i=1}^{r} b_i f(t_n+c_i h, Q_i),\quad n\in \mathbb{N},
    \end{alignat}
\end{subequations}
and as shown in \cite{hairernw93sod}, under the assumption
\begin{equation}
    \bar{b}_i=b_i(1-c_i),\quad i=1,\ldots, r,
\end{equation}
the number of order conditions are drastically reduced and there are
rather fewer order conditions need to be further considered, as
listed below \cite{hairernw93sod}
\begin{equation*}
\begin{array}{|l|l|}
1)  \sum\limits_{i=1}^{r} b_i=1; & 2) \sum\limits_{i=1}^{r} b_ic_i=\frac{1}{2}; \\
3) \sum\limits_{i=1}^{r} b_ic_i^2=\frac{1}{3}; & 4)
\sum\limits_{i=1}^{r}\sum\limits_{j=1}^{r} b_i
\bar{a}_{ij}=\frac{1}{6}; \\
5) \sum\limits_{i=1}^{r} b_ic_i^3=\frac{1}{4}; & 6)
\sum\limits_{i=1}^{r}\sum\limits_{j=1}^{r} b_ic_i
\bar{a}_{ij}=\frac{1}{8}; \\
7) \sum\limits_{i=1}^{r}\sum\limits_{j=1}^{r} b_i
\bar{a}_{ij}c_j=\frac{1}{24}; & 8) \sum\limits_{i=1}^{r}
b_ic_i^4=\frac{1}{5}; \\
9) \sum\limits_{i=1}^{r}\sum\limits_{j=1}^{r} b_i
c_i^2\bar{a}_{ij}=\frac{1}{10}; & 10)
\sum\limits_{i=1}^{r}\sum\limits_{j=1}^{r}\sum\limits_{k=1}^{r}
b_i \bar{a}_{ij}\bar{a}_{ik}=\frac{1}{20}; \\
11) \sum\limits_{i=1}^{r}\sum\limits_{j=1}^{r} b_i
c_i\bar{a}_{ij}c_j=\frac{1}{30};  & 12)
\sum\limits_{i=1}^{r}\sum\limits_{j=1}^{r} b_i
\bar{a}_{ij}c_j^2=\frac{1}{60}; \\
13) \sum\limits_{i=1}^{r}\sum\limits_{j=1}^{r}\sum\limits_{k=1}^{r}
b_i \bar{a}_{ij}\bar{a}_{jk}=\frac{1}{120};  & 14)  \cdots
\end{array}
\end{equation*}

If the condition 1) holds, then the RKN method is of order 1; if
conditions 1)-2) hold, then the RKN method is of order 2; if
conditions 1)-4) hold, then the RKN method is of order 3; if
conditions 1)-7) hold, then the RKN method is of order 4; if
conditions 1)-13) hold, then the RKN method is of order 5.

Similarly as the classical case, under the assumption
$\bar{B}_\tau=\hat{B}_\tau(1-C_\tau)$, we have the following order
conditions for csRKN method
\begin{equation*}
\begin{array}{|l|l|}
1)  \int_0^1 \hat{B}_\tau  d\tau=1; & 2) \int_0^1 \hat{B}_\tau
C_\tau  d\tau=\frac{1}{2}; \\
3) \int_0^1 \hat{B}_\tau C_\tau^2 d\tau=\frac{1}{3}; & 4)
\int_0^1\int_0^1
\hat{B}_\tau\bar{A}_{\tau,\si}d\si d\tau=\frac{1}{6}; \\
5) \int_0^1 \hat{B}_\tau C_\tau^3 d\tau=\frac{1}{4}; & 6)
\int_0^1\int_0^1 \hat{B}_\tau C_\tau \bar{A}_{\tau,\si}d\si
d\tau=\frac{1}{8}; \\
7) \int_0^1\int_0^1 \hat{B}_\tau \bar{A}_{\tau,\si}C_\si d\si
d\tau=\frac{1}{24}; & 8) \int_0^1
\hat{B}_\tau  C_\tau^4 d\tau=\frac{1}{5}; \\
9) \int_0^1\int_0^1 \hat{B}_\tau C_\tau^2 \bar{A}_{\tau,\si}d\si
d\tau=\frac{1}{10}; & 10) \int_0^1\int_0^1\int_0^1
\hat{B}_\tau\bar{A}_{\tau,\si}
\bar{A}_{\tau,\rho}d\rho d\si d\tau=\frac{1}{20};  \\
11) \int_0^1\int_0^1 \hat{B}_\tau C_\tau \bar{A}_{\tau,\si}C_\si
d\si d\tau=\frac{1}{30};  & 12) \int_0^1\int_0^1 \hat{B}_\tau
\bar{A}_{\tau,\si}C_\si^2 d\si
d\tau=\frac{1}{60}; \\
13) \int_0^1\int_0^1\int_0^1
\hat{B}_\tau\bar{A}_{\tau,\si}\bar{A}_{\si,\rho}d\rho d\si
d\tau=\frac{1}{120};  & 14)  \cdots
\end{array}
\end{equation*}

Note that $\bar{B}_\tau=\hat{B}_\tau(1-C_\tau)$ is naturally
satisfied by the hypothesis $\hat{B}_\tau=1,\,C_\tau=\tau$.
Moreover, if the condition 1) holds, then the csRKN method is of
order 1; if conditions 1)-2) hold, then the csRKN method is of order
2; if conditions 1)-4) hold, then the csRKN method is of order 3; if
conditions 1)-7) hold, then the csRKN method is of order 4; if
conditions 1)-13) hold, then the csRKN method is of order 5.

It is obvious that the conditions 1)-2) always hold, so the csRKN
methods presented in this paper are of order 2 at least. Moreover,
it is found that conditions 2), 3), 5), 8) are also naturally
satisfied, so other conditions will be further investigated when we
try to construct higher order csRKN methods.

\subsection{Construction of symplectic RKN methods}

Though it is possible to construct RKN type methods with arbitrarily
high order for some special purposes, here we will restrict
ourselves on the construction of symplectic integrators. By using
expansion of Legendre orthogonal polynomials, we have the following
identities
\begin{equation}
\begin{split}
&1=P_0(\tau),\\
&\tau=\frac{1}{2}P_0(\tau)+\frac{\sqrt{3}}{6}P_1(\tau),\\
&\tau^2=\frac{1}{3}P_0(\tau)+\frac{\sqrt{3}}{6}P_1(\tau)+\frac{\sqrt{5}}{30}P_2(\tau),\\
&\tau^3=\frac{1}{4}P_0(\tau)+\frac{3\sqrt{3}}{20}P_1(\tau)+
\frac{\sqrt{5}}{20}P_2(\tau)+\frac{\sqrt{7}}{140}P_3(\tau),\\
&\cdots
\end{split}
\end{equation}
which turn out to be very helpful for the investigation of the order
conditions. For convenience, here we provide the former several
Legendre polynomials
\begin{equation}
\begin{split}
&P_0(\tau)=1,\\
&P_1(\tau)=\sqrt{3}(2\tau-1),\\
&P_2(\tau)=\sqrt{5}(6\tau^2-6\tau+1),\\
&P_3(\tau)=\sqrt{7}(20\tau^3-30\tau^2+12\tau-1),\\
&\cdots
\end{split}
\end{equation}

In what follows, we present the construction of symplectic
integrators up to order 5 and some examples will be given.

\subsubsection{$2$-order symplectic integrators}

Although the csRKN methods presented in this paper are always of
order 2 at least, but for a symplectic csRKN method, the coefficient
$\bar{A}_{\tau,\si}$ should be designed by Theorem
\ref{construct_scsRKN}, an example is given in what follows.
\begin{table}
\[\ba{c|c} \frac{1}{2} & \alpha\\[2pt]
\hline & \frac{1}{2}\\[2pt]
\hline & 1\ea\qquad \ba{c|cc} 0 &
\frac{1}{4}\alpha-\frac{\sqrt{3}}{2}\beta+\frac{1}{8}&
\frac{3}{4}\alpha-\frac{\sqrt{3}}{2}\beta-\frac{1}{8}\\[2pt]
\frac{2}{3} & \frac{1}{4}\alpha-\frac{\sqrt{3}}{6}\beta+\frac{1}{8}&
\frac{3}{4}\alpha+\frac{\sqrt{3}}{2}\beta-\frac{1}{8}\\[2pt]
\hline & \frac{1}{4}& \frac{1}{4}\\[2pt]
\hline & \frac{1}{4}& \frac{3}{4}
\ea\]\\
\[\ba{c|cc} \frac{1}{3} &
\frac{3}{4}\alpha-\frac{\sqrt{3}}{2}\beta+\frac{1}{8}&
\frac{1}{4}\alpha+\frac{\sqrt{3}}{6}\beta-\frac{1}{8}\\[2pt]
1 & \frac{3}{4}\alpha+\frac{\sqrt{3}}{2}\beta+\frac{1}{8}&
\frac{1}{4}\alpha+\frac{\sqrt{3}}{2}\beta-\frac{1}{8}\\[2pt]
\hline & \frac{1}{2}& 0\\[2pt]
\hline & \frac{3}{4}& \frac{1}{4} \ea\qquad \ba{c|cc} 0 &
\frac{1}{2}\alpha-\sqrt{3}\beta+\frac{1}{4}&
\frac{1}{2}\alpha-\frac{1}{4}\\[2pt]
1 & \frac{1}{2}\alpha+\frac{1}{4}&
\frac{1}{2}\alpha+\sqrt{3}\beta-\frac{1}{4}\\[2pt]
\hline & \frac{1}{2}& 0\\[2pt]
\hline & \frac{1}{2}& \frac{1}{2} \ea\] \caption{Symplectic RKN
methods of order 2 by using different quadrature formulae. Top Left:
by Gaussian quadrature, Top Right: by Radau-left quadrature, Bottom
Left: by Radau-right quadrature, Bottom Right: by Lobatto
quadrature.}\label{exa:SRKN2}
\end{table}
\begin{exa}
If we take the coefficients
$(\bar{A}_{\tau,\si},\,\bar{B}_\tau,\,\hat{B}_\tau,\,C_\tau)$ as
\begin{equation} \label{eq:2coeff}
\begin{split}
&\bar{A}_{\tau,\si}=\alpha+(\beta-\frac{\sqrt{3}}{6})
P_1(\si)+\beta P_1(\tau),\\
&\bar{B}_\tau=1-\tau,\;\hat{B}_\tau=1,\;C_\tau=\tau,
\end{split}
\end{equation}
where $\alpha$ and $\beta$ are two real parameters, then we get a
two-parameter family of $2$-order symplectic csRKN methods.

By using any numerical quadrature formula with order $p\geq2$ we can
get the classical symplectic RKN methods of order\footnote{This can
be easily checked by the classical order conditions that listed in
subsection 4.1.} $2$, e.g., Gaussian quadrature with $1$ node,
Radau-left or Radau-right quadrature with $2$ nodes, Lobatto
quadrature with $2$ nodes. The corresponding symplectic RKN methods
obtained by using different quadrature formulae are shown in Table
\ref{exa:SRKN2}.
\end{exa}

\subsubsection{$3$-order symplectic integrators}

By inserting \eqref{sym_cond} into the fourth order condition and
using the orthogonality of the Legendre polynomials, it gives
\begin{equation}
\int_0^1\int_0^1 \hat{B}_\tau\bar{A}_{\tau,\si}d\si
d\tau=\int_0^1(\alpha_{(0,0)}+\alpha_{(1,0)}P_1(\tau)
+\sum\limits_{i>1}\alpha_{(i,0)}P_i(\tau))d\tau=
\alpha_{(0,0)}=\frac{1}{6},
\end{equation}
therefore, if we require $\alpha_{(0,0)}=\frac{1}{6}$, then we can
get a class of $3$-order symplectic csRKN methods.
\begin{table}
\[\ba{c|cc} \frac{3-\sqrt{3}}{6}&
\frac{1+\sqrt{3}}{12}-\alpha+\frac{1}{2}\beta&
\frac{1-\sqrt{3}}{12}-\frac{1}{2}\beta\\[2pt]
\frac{3+\sqrt{3}}{6} & \frac{1+\sqrt{3}}{12}-\frac{1}{2}\beta&
\frac{1-\sqrt{3}}{12}+\alpha+\frac{1}{2}\beta\\[2pt]
\hline & \frac{1}{4}+\frac{\sqrt{3}}{12}& \frac{1}{4}-\frac{\sqrt{3}}{12}\\[2pt]
\hline & \frac{1}{2}& \frac{1}{2} \ea\qquad \ba{c|cc} 0 &
\frac{2-6\sqrt{3}\alpha+9\beta}{12}&
-\frac{\sqrt{3}}{2}\alpha-\frac{3}{4}\beta\\[2pt]
\frac{2}{3} & \frac{2-2\sqrt{3}\alpha-3\beta}{12} &
\frac{\sqrt{3}}{2}\alpha+\frac{1}{4}\beta\\[2pt]
\hline & \frac{1}{4}& \frac{1}{4}\\[2pt]
\hline & \frac{1}{4}& \frac{3}{4}
\ea\]\\
\[\ba{c|cc} \frac{1}{3} &
\frac{1-2\sqrt{3}\alpha+\beta}{4} &
\frac{-1+2\sqrt{3}\alpha-3\beta}{12}\\[2pt]
1 & \frac{1+2\sqrt{3}\alpha-3\beta}{4}&
\frac{-1+6\sqrt{3}\alpha+9\beta}{12}\\[2pt]
\hline & \frac{1}{2}& 0\\[2pt]
\hline & \frac{3}{4}& \frac{1}{4} \ea\qquad \ba{c|ccc} 0 &
\frac{6-18\sqrt{3}\alpha+27\beta}{54} &
\frac{1-6\sqrt{3}\alpha}{9} & -\frac{1}{18}-\frac{1}{2}\beta \\[2pt]
\frac{1}{2} & \frac{1}{9}-\frac{\sqrt{3}}{6}\alpha  &
\frac{1}{9}& -\frac{1}{18}+\frac{\sqrt{3}}{6}\alpha\\[2pt]
1 & \frac{1}{9}-\frac{1}{2}\beta& \frac{1+6\sqrt{3}\alpha}{9} &
\frac{-1+6\sqrt{3}\alpha+9\beta}{18}\\[2pt]
\hline & \frac{1}{6}& \frac{1}{3}& 0\\[2pt]
\hline & \frac{1}{6}& \frac{2}{3}& \frac{1}{6} \ea\]
\caption{Symplectic RKN methods of order 3 by using different
quadrature formulae. Top Left: by Gaussian quadrature, Top Right: by
Radau-left quadrature, Bottom Left: by Radau-right quadrature,
Bottom Right: by Lobatto quadrature.}\label{exa:SRKN3}
\end{table}
\begin{exa}
If we take the coefficients
$(\bar{A}_{\tau,\si},\,\bar{B}_\tau,\,\hat{B}_\tau,\,C_\tau)$ as
\begin{equation}\label{eq:3coeff}
\begin{split}
&\bar{A}_{\tau,\si}=\frac{1}{6}+(\alpha-\frac{\sqrt{3}}{6})
P_1(\si)+\alpha P_1(\tau)+\beta P_1(\tau)P_1(\si),\\
&\bar{B}_\tau=1-\tau,\;\hat{B}_\tau=1,\;C_\tau=\tau,
\end{split}
\end{equation}
then we get a two-parameter family of $3$-order symplectic csRKN
methods.

By using any numerical quadrature formula with order $p\geq3$ we can
get the classical symplectic RKN methods of order $3$, e.g.,
Gaussian quadrature with $2$ nodes, Radau-left or Radau-right
quadrature with $2$ nodes, Lobatto quadrature with $3$ nodes.  The
corresponding symplectic RKN methods obtained by using different
quadrature formulae are shown in Table \ref{exa:SRKN3}.
\end{exa}

\subsubsection{$4$-order symplectic integrators}

By inserting \eqref{sym_cond} with $\alpha_{(0,0)}=\frac{1}{6}$ into
the sixth order condition and using the orthogonality of the
Legendre polynomials, it gives
\begin{equation}
\begin{split}
&\int_0^1\int_0^1 \hat{B}_\tau C_\tau \bar{A}_{\tau,\si}d\si
d\tau\\
&=\int_0^1\int_0^1 \tau \bar{A}_{\tau,\si}d\si
d\tau\\
&=\int_0^1
(\frac{1}{2}P_0(\tau)+\frac{\sqrt{3}}{6}P_1(\tau))(\int_0^1\bar{A}_{\tau,\si}d\si)
d\tau\\
&=\int_0^1(\frac{1}{2}P_0(\tau)+\frac{\sqrt{3}}{6}P_1(\tau))(\frac{1}{6}+\alpha_{(1,0)}P_1(\tau)+
\sum\limits_{i>1}\alpha_{(i,0)}P_i(\tau))d\tau\\
&=\alpha_{(1,0)}\frac{\sqrt{3}}{6}+\frac{1}{12}=\frac{1}{8},
\end{split}
\end{equation}
which provides $\alpha_{(1,0)}=\frac{\sqrt{3}}{12}$.

Similarly, by exploiting the seventh order condition we obtain
$\alpha_{(0,1)}=-\frac{\sqrt{3}}{12}$, which coincides with the
symplectic condition
$\alpha_{(0,1)}-\alpha_{(1,0)}=-\frac{\sqrt{3}}{6}$ that given in
Theorem \ref{construct_scsRKN}. Therefore, if we require
$$\alpha_{(0,0)}=\frac{1}{6},\;\alpha_{(1,0)}=\frac{\sqrt{3}}{12},\;\alpha_{(0,1)}=-\frac{\sqrt{3}}{12},$$
then we can get a class of $4$-order symplectic csRKN methods.

\begin{table}
\[\ba{c|cc} \frac{3-\sqrt{3}}{6}&
\frac{1}{12}+\frac{1}{2}\alpha&
\frac{1-\sqrt{3}}{12}-\frac{1}{2}\alpha\\[2pt]
\frac{3+\sqrt{3}}{6} & \frac{1+\sqrt{3}}{12}-\frac{1}{2}\alpha&
\frac{1}{12}+\frac{1}{2}\alpha\\[2pt]
\hline & \frac{1}{4}+\frac{\sqrt{3}}{12}& \frac{1}{4}-\frac{\sqrt{3}}{12}\\[2pt]
\hline & \frac{1}{2}& \frac{1}{2} \ea \]\\
\small{\[\ba{c|ccc}0 & \frac{1+18\alpha}{54} &
\frac{-11+4\sqrt{6}+(-36+54\sqrt{6})\alpha}{216}&
\frac{-11-4\sqrt{6}+(-36-54\sqrt{6})\alpha}{216} \\[2pt]
\frac{6-\sqrt{6}}{10} &
\frac{28-3\sqrt{6}}{540}+\frac{(-15+15\sqrt{6})\alpha}{225}
&\frac{16+\sqrt{6}}{216}+\frac{(12-3\sqrt{6})\alpha}{36}
& \frac{98-53\sqrt{6}}{1080}+\frac{(-240+15\sqrt{6})\alpha}{900}\\[2pt]
\frac{6+\sqrt{6}}{10}
&\frac{28+3\sqrt{6}}{540}+\frac{(-15-15\sqrt{6})\alpha}{225}&
\frac{98+53\sqrt{6}}{1080}-\frac{(240+15\sqrt{6})\alpha}{900}&
\frac{16-\sqrt{6}}{216}+\frac{(12+3\sqrt{6})\alpha}{36}\\[2pt]
\hline & \frac{1}{9}& \frac{7+2\sqrt{6}}{36}& \frac{7-2\sqrt{6}}{36}\\[2pt]
\hline & \frac{1}{9}& \frac{16+\sqrt{6}}{36}&
\frac{16-\sqrt{6}}{36}\ea\]\\
\[\ba{c|ccc} \frac{4-\sqrt{6}}{10} & \frac{16-\sqrt{6}}{216}+\frac{(12+3\sqrt{6})\alpha}{36}&
\frac{62-43\sqrt{6}}{1080}-\frac{(240+15\sqrt{6})\alpha}{900}&
-\frac{8+3\sqrt{6}}{540}+\frac{-(15+15\sqrt{6})\alpha}{225}\\[2pt]
\frac{4+\sqrt{6}}{10} &
\frac{62+43\sqrt{6}}{1080}-\frac{(240-15\sqrt{6})\alpha}{900}&
\frac{16+\sqrt{6}}{216}+\frac{(12-3\sqrt{6})\alpha}{36}&
-\frac{8-3\sqrt{6}}{540}+\frac{-(15-15\sqrt{6})\alpha}{225}\\[2pt]
1 & \frac{43+2\sqrt{6}}{216}+\frac{-(6+9\sqrt{6})\alpha}{36}&
\frac{43-2\sqrt{6}}{216}+\frac{-(6-9\sqrt{6})\alpha}{36}& \frac{1+18\alpha}{54}\\[2pt]
\hline & \frac{9+\sqrt{6}}{36}& \frac{9-\sqrt{6}}{36}& 0\\[2pt]
\hline & \frac{16-\sqrt{6}}{36}& \frac{16+\sqrt{6}}{36}& \frac{1}{9}
\ea\]}\\
\[ \ba{c|ccc} 0 &
\frac{1+18\alpha+12\sqrt{5}\beta}{36} &
\frac{-1+6\sqrt{5}\beta}{18} & \frac{-2-18\alpha+12\sqrt{5}\beta}{36} \\[2pt]
\frac{1}{2} & \frac{5+6\sqrt{5}\beta}{72}  &
\frac{1-6\sqrt{5}\beta}{9}&\frac{-1+6\sqrt{5}\beta}{72}\\[2pt]
1 & \frac{2-9\alpha+6\sqrt{5}\beta}{18}& \frac{5+6\sqrt{5}\beta}{18}
&\frac{1+18\alpha+12\sqrt{5}\beta}{36}\\[2pt]
\hline & \frac{1}{6}& \frac{1}{3}& 0\\[2pt]
\hline & \frac{1}{6}& \frac{2}{3}& \frac{1}{6} \ea \]
\caption{Symplectic RKN methods of order 4 by using different
quadrature formulae. First: by Gaussian quadrature, Second: by
Radau-left quadrature ($\beta=0$), Third: by Radau-right quadrature
($\beta=0$), Fourth: by Lobatto quadrature.}\label{exa:SRKN4}
\end{table}

\begin{exa}
If we take the coefficients
$(\bar{A}_{\tau,\si},\,\bar{B}_\tau,\,\hat{B}_\tau,\,C_\tau)$ as
\begin{equation}\label{eq:4coeff}
\begin{split}
&\bar{A}_{\tau,\si}=\frac{1}{6}+\frac{\tau-\sigma}{2}+\alpha
P_1(\tau)P_1(\si)+\beta(P_2(\tau)+P_2(\si)),\\
&\bar{B}_\tau=1-\tau,\;\hat{B}_\tau=1,\;C_\tau=\tau,
\end{split}
\end{equation}
then we get a two-parameter family of $4$-order symplectic csRKN
methods.

By using any numerical quadrature formula with order $p\geq4$ we can
get the classical symplectic RKN methods of order $4$, e.g.,
Gaussian quadrature with $2$ nodes, Radau-left or Radau-right
quadrature with $3$ nodes, Lobatto quadrature with $3$ nodes.  The
corresponding symplectic RKN methods obtained by using different
quadrature formulae are shown in Table \ref{exa:SRKN4}. Note that,
actually, more free parameters can be taken into account.
\end{exa}

\subsubsection{$5$-order symplectic integrators}

By the previous discussions we have obtained that
$\alpha_{(0,0)}=\frac{1}{6},\,\alpha_{(1,0)}=\frac{\sqrt{3}}{12},
\,\alpha_{(0,1)}=-\frac{\sqrt{3}}{12}$, now we shall insert
\eqref{sym_cond} into the remaining order conditions.

For condition 9): We compute
\begin{equation*}
\begin{split}
&\int_0^1\int_0^1 \hat{B}_\tau C_\tau^2 \bar{A}_{\tau,\si}d\si
d\tau\\
&=\int_0^1\int_0^1 \tau^2 \bar{A}_{\tau,\si}d\si
d\tau\\
&=\int_0^1
(\frac{1}{3}P_0(\tau)+\frac{\sqrt{3}}{6}P_1(\tau)+
\frac{\sqrt{5}}{30}P_2(\tau))(\int_0^1\bar{A}_{\tau,\si}d\si)
d\tau\\
&=\int_0^1(\frac{1}{3}P_0(\tau)+\frac{\sqrt{3}}{6}P_1(\tau)
+\frac{\sqrt{5}}{30}P_2(\tau))(\alpha_{(0,0)}+\alpha_{(1,0)}P_1(\tau)+
\sum\limits_{i>1}\alpha_{(i,0)}P_i(\tau))d\tau\\
&=\frac{1}{3}\alpha_{(0,0)}+\frac{\sqrt{3}}{6}\alpha_{(1,0)}
+\frac{\sqrt{5}}{30}\alpha_{(2,0)}=\frac{1}{10},
\end{split}
\end{equation*}
which then gives $\alpha_{(2,0)}=\frac{\sqrt{5}}{60}$.

For condition 10): Since
\begin{equation*}
\begin{split}
&\int_0^1\int_0^1\int_0^1
\hat{B}_\tau\bar{A}_{\tau,\si}\bar{A}_{\tau,\rho}d\rho d\si
d\tau\\
&=\int_0^1(\int_0^1
\bar{A}_{\tau,\si}d\si)(\int_0^1\bar{A}_{\tau,\rho}d\rho)d\tau\\
&=\int_0^1(\alpha_{(0,0)}+\alpha_{(1,0)}P_1(\tau)+
\sum\limits_{i>1}\alpha_{(i,0)}P_i(\tau))^2d\tau\\
&=\alpha_{(0,0)}^2+\alpha_{(1,0)}^2+
\sum\limits_{i>1}\alpha_{(i,0)}^2=\frac{1}{20},
\end{split}
\end{equation*}
substituting the values of
$\alpha_{(0,0)},\,\alpha_{(1,0)},\,\alpha_{(2,0)}$ into it, then we
get $\sum\limits_{i>2}\alpha_{(i,0)}^2=0$ which means
$\alpha_{(i,0)}=0$ for all $i>2$.

For condition 11):  Since
\begin{equation*}
\begin{split}
&\int_0^1\int_0^1 \hat{B}_\tau C_\tau
\bar{A}_{\tau,\si}C_\si d\si d\tau\\
&=\int_0^1 \tau\big(\int_0^1\bar{A}_{\tau,\si}\si d\si\big) d\tau\\
&=\int_0^1(\frac{1}{2}P_0(\tau)+\frac{\sqrt{3}}{6}P_1(\tau))
\big[\int_0^1\Big(\alpha_{(0,0)}+\alpha_{(0,1)}P_1(\si)+\alpha_{(1,0)}P_1(\tau)\\
&\quad+\sum\limits_{i+j>1}\alpha_{(i,j)}P_i(\tau)P_j(\si)\Big)\big(\frac{1}{2}P_0(\si)
+\frac{\sqrt{3}}{6}P_1(\si)\big)d\si\big]d\tau\\
&=\int_0^1(\frac{1}{2}P_0(\tau)+\frac{\sqrt{3}}{6}P_1(\tau))
\big(\frac{1}{2}\alpha_{(0,0)}+\frac{1}{2}\alpha_{(1,0)}P_1(\tau)
+\frac{1}{2}\sum\limits_{i>1}\alpha_{(i,0)}P_i(\tau)\\
&\quad+\frac{\sqrt{3}}{6}\alpha_{(0,1)}+\frac{\sqrt{3}}{6}\sum\limits_{i>0}
\alpha_{(i,1)}P_i(\tau)\big) d\tau\\
&=\frac{1}{4}\alpha_{(0,0)}+\frac{\sqrt{3}}{12}\alpha_{(0,1)}
+\frac{\sqrt{3}}{12}\alpha_{(1,0)}+\frac{1}{12}\alpha_{(1,1)}=\frac{1}{30},
\end{split}
\end{equation*}
this gives $\alpha_{(1,1)}=-\frac{1}{10}$.

For condition 12):  By the very similar deduction as that for (9) we
get
\begin{equation}
\frac{1}{3}\alpha_{(0,0)}+\frac{\sqrt{3}}{6}\alpha_{(0,1)}+
\frac{\sqrt{5}}{30}\alpha_{(0,2)}=\frac{1}{60},
\end{equation}
which provides $\alpha_{(0,2)}=\frac{\sqrt{5}}{60}$.

For condition 13):  We have
\begin{equation*}
\begin{split}
&\int_0^1\int_0^1\int_0^1
\hat{B}_\tau\bar{A}_{\tau,\si}\bar{A}_{\si,\rho}d\rho d\si
d\tau\\
&=\int_0^1 (\int_0^1\bar{A}_{\tau,\si} d\tau)
(\int_0^1\bar{A}_{\si,\rho}d\rho)d\si\\
&=\int_0^1\Big(\alpha_{(0,0)}+\alpha_{(0,1)}P_1(\si)+
\sum\limits_{j>1}\alpha_{(0,j)}P_j(\si)\Big)
\Big(\alpha_{(0,0)}+\alpha_{(1,0)}P_1(\si)\\
&\quad+\sum\limits_{i>1}\alpha_{(i,0)}P_i(\si)\Big)d\si\\
&=\alpha_{(0,0)}^2+\alpha_{(0,1)}\alpha_{(1,0)}+
\sum\limits_{j>1}\alpha_{(0,j)}\alpha_{(j,0)}=\frac{1}{120},
\end{split}
\end{equation*}
After substituting the values of
$\alpha_{(0,0)},\,\alpha_{(0,1)},\,\alpha_{(1,0)},\,\alpha_{(2,0)}$
into it, it brings out that
$\sum\limits_{j>2}\alpha_{(0,j)}\alpha_{(j,0)}=0$. In addition, by
Theorem \ref{construct_scsRKN} we have
$\alpha_{(0,j)}=\alpha_{(j,0)}$ for $j>2$, then it follows that
$$\alpha_{(0,j)}=\alpha_{(j,0)}=0,\quad j>2.$$

In summary, for obtaining a symplectic csRKN method of order 5, we
should require that
\begin{equation}\label{coef: fifth}
\begin{split}
&\alpha_{(0,0)}=\frac{1}{6},\,\alpha_{(1,0)}=\frac{\sqrt{3}}{12},
\,\alpha_{(0,1)}=-\frac{\sqrt{3}}{12},\\
&\alpha_{(1,1)}=-\frac{1}{10},\,\alpha_{(2,0)}=\alpha_{(0,2)}
=\frac{\sqrt{5}}{60},\\
&\alpha_{(0,j)}=\alpha_{(j,0)}=0,\quad j>2,
\end{split}
\end{equation}
and other parameters $\alpha_{(i,j)}$ can be freely assigned.

\begin{table}
\[\ba{c|ccc} \frac{5-\sqrt{15}}{10} & \frac{2-90\alpha+60\beta}{135} &
\frac{19-6\sqrt{15}+180\alpha-120\beta}{270}&
\frac{62-15\sqrt{15}+120\beta}{540}\\[2pt]
\frac{1}{2} & \frac{19+6\sqrt{15}+180\alpha-120\beta}{432}&
\frac{1+15\beta}{27}&
\frac{19-6\sqrt{15}-180\alpha-120\beta}{432}\\[2pt]
\frac{5+\sqrt{15}}{10}& \frac{62+15\sqrt{15}+120\beta}{540}
&\frac{19+6\sqrt{15}-180\alpha-120\beta}{270}&\frac{2+90\alpha+60\beta}{135}\\[2pt]
\hline & \frac{5+\sqrt{15}}{36} & \frac{2}{9} & \frac{5-\sqrt{15}}{36}\\[2pt]
\hline & \frac{5}{18} & \frac{4}{9} & \frac{5}{18}
\ea\]\\
{\scriptsize\[\ba{c|ccc}0 & \frac{1-60\sqrt{15}\alpha}{270} &
\frac{-4-19\sqrt{6}+(240\sqrt{15}-180\sqrt{10})\alpha}{2160}&
\frac{-4+19\sqrt{6}+(240\sqrt{15}+180\sqrt{10})\alpha}{2160}\\[2pt]
\frac{6-\sqrt{6}}{10} &
\frac{181-36\sqrt{6}+(84\sqrt{15}-72\sqrt{10})\alpha}{2700}
&\frac{17+2\sqrt{6}+60\sqrt{15}\alpha}{540}
&\frac{301-136\sqrt{6}-(384\sqrt{15}-72\sqrt{10})\alpha}{2700}\\[2pt]
\frac{6+\sqrt{6}}{10}
&\frac{181+36\sqrt{6}+(84\sqrt{15}+72\sqrt{10})\alpha}{2700}&
\frac{301+136\sqrt{6}-(384\sqrt{15}+72\sqrt{10})\alpha}{2700}&
\frac{17-2\sqrt{6}+60\sqrt{15}\alpha}{540}\\[2pt]
\hline & \frac{1}{9}& \frac{7+2\sqrt{6}}{36}& \frac{7-2\sqrt{6}}{36}\\[2pt]
\hline & \frac{1}{9}& \frac{16+\sqrt{6}}{36}&
\frac{16-\sqrt{6}}{36}\ea\]\\
\[\ba{c|ccc} \frac{4-\sqrt{6}}{10} &
\frac{17-2\sqrt{6}-60\sqrt{15}\alpha}{540}&
\frac{211-104\sqrt{6}+(384\sqrt{15}+72\sqrt{10})\alpha}{2700}&
\frac{1+6\sqrt{6}-(84\sqrt{15}+72\sqrt{10})\alpha}{2700}\\[2pt]
\frac{4+\sqrt{6}}{10} &
\frac{211+104\sqrt{6}+(384\sqrt{15}-72\sqrt{10})\alpha}{2700}&
\frac{17+2\sqrt{6}-60\sqrt{15}\alpha}{540}&
\frac{1-6\sqrt{6}-(84\sqrt{15}-72\sqrt{10})\alpha}{2700}\\[2pt]
1 & \frac{536+79\sqrt{6}-(240\sqrt{15}+180\sqrt{10})\alpha}{2160}
&\frac{536-79\sqrt{6}-(240\sqrt{15}-180\sqrt{10})\alpha}{2160}
&\frac{1+60\sqrt{15}\alpha}{270}\\[2pt]
\hline & \frac{9+\sqrt{6}}{36}& \frac{9-\sqrt{6}}{36}& 0\\[2pt]
\hline & \frac{16-\sqrt{6}}{36}& \frac{16+\sqrt{6}}{36}& \frac{1}{9}
\ea\]} \caption{Symplectic RKN methods of order 5 by using different
quadrature formulae. First: by Gaussian quadrature, Second: by
Radau-left quadrature ($\beta=0$), Third: by Radau-right quadrature
($\beta=0$).}\label{exa:SRKN5}
\end{table}

\begin{exa}
If we take the coefficients
$(\bar{A}_{\tau,\si},\,\bar{B}_\tau,\,\hat{B}_\tau,\,C_\tau)$ as
\begin{equation}\label{eq:5coeff}
\begin{split}
&\bar{A}_{\tau,\si}=\sum\limits_{i+j\leq2}\alpha_{(i,j)}P_i(\tau)P_j(\sigma)+
\alpha \Big(P_1(\tau)P_2(\sigma)+P_2(\tau)P_1(\sigma)\Big)+\beta P_2(\tau)P_2(\sigma),\\
&\qquad\qquad\qquad\bar{B}_\tau=1-\tau,\;\hat{B}_\tau=1,\;C_\tau=\tau,
\end{split}
\end{equation}
where $\alpha_{(i,j)},\,i+j\leq2$ satisfy \eqref{coef: fifth} and
$\alpha,\,\beta$ are real parameters, then we get a two-parameter
family of $5$-order symplectic csRKN methods.

By using any numerical quadrature formula with order $p\geq5$ we can
get the classical symplectic RKN methods of order $5$, e.g.,
Gaussian quadrature with $3$ nodes, Radau-left or Radau-right
quadrature with $3$ nodes, Lobatto quadrature with $4$ nodes. The
corresponding symplectic RKN methods obtained by using different
quadrature formulae are shown in Table \ref{exa:SRKN5}.

Notice that based on 4-nodes Lobatto qudrature,  the $4$-stage
$5$-order symplectic RKN method with coefficients denoted by
$(\bar{A},\,\bar{b},\,b,\,c)$ is too lengthy to be shown in a
Butcher tableau, we present it as follows in use of Matlab notations
\begin{equation*}
  \begin{split}
\bar{A} =&[ \frac{1-60\sqrt{15}\alpha+150\beta}{360}, \frac{-5-3\sqrt{5}-(300\sqrt{3}-60\sqrt{15})\alpha-300\beta}{720},\\
&\frac{-5+3\sqrt{5}+(300\sqrt{3}+60\sqrt{15})\alpha-300\beta}{720}, \frac{2+75\beta}{180};\\
&\frac{29}{720}-\frac{11\sqrt{5}+(100\sqrt{3}-20\sqrt{15})\alpha+100\beta}{1200}, \frac{11+60\sqrt{3}\alpha+30\beta}{360},\\
&\frac{29-15\sqrt{5}+30\beta}{360}, -\frac{1}{720}+\frac{\sqrt{5}-(20\sqrt{15}+100\sqrt{3})\alpha-100\beta}{1200};\\
&\frac{29}{720}+\frac{11\sqrt{5}+(100\sqrt{3}+20\sqrt{15})\alpha-100\beta}{1200}, \frac{29+15\sqrt{5}+30\beta}{360},\\
&\frac{11-60\sqrt{3}\alpha+30\beta}{360}, -\frac{1}{720}-\frac{\sqrt{5}+(20\sqrt{15}-100\sqrt{3})\alpha+100\beta}{1200};\\
& \frac{17+75\beta}{180}, \frac{145+33\sqrt{5}-(60\sqrt{15}+300\sqrt{3})\alpha-300\beta}{720},\\
&\frac{145-33\sqrt{5}-(60\sqrt{15}-300\sqrt{3})\alpha-300\beta}{720}, \frac{1+60\sqrt{15}\alpha+150\beta}{360}],\\
\bar{b}=&[\frac{1}{12}, \frac{5+\sqrt{5}}{24},
\frac{5-\sqrt{5}}{24}, 0], \,b=[\frac{1}{12}, \frac{5}{12},
\frac{5}{12}, \frac{1}{12}], \,c=[0, \frac{5-\sqrt{5}}{10},
\frac{5+\sqrt{5}}{10}, 1].
  \end{split}
\end{equation*}
\end{exa}

\section{Diagonally implicit symplectic RKN methods}

It is known that diagonally implicit methods are more attractive
than fully implicit methods for the sake of time cost savings and
high efficiency in the numerical computations. A diagonally implicit
RKN method is a method (\ref{eq:grkn1}-\ref{eq:grkn3}) with
coefficient $(\bar{a}_{ij},\,\bar{b}_i,\,b_i,\,c_i)$ satisfying
\begin{equation}\label{diag imp:condi}
\bar{a}_{ij}=0,\,j> i.
\end{equation}
By setting more parameters and solving a linear algebraic system, it
is possible to get diagonally implicit symplectic integrators.

For example, let us take
\begin{equation} \label{diag imp:2coeff}
\begin{split}
&\bar{A}_{\tau,\si}=\alpha+(\beta-\frac{\sqrt{3}}{6})
P_1(\si)+\beta P_1(\tau)+\gamma P_1(\tau)P_1(\si),\\
&\bar{B}_\tau=1-\tau,\;\hat{B}_\tau=1,\;C_\tau=\tau,
\end{split}
\end{equation}
where three real parameters $\alpha,\,\beta,\,\gamma$ are included.
In such a case, by using Radau-left, Radau-right and Lobatto
quadrature with 2 nodes, respectively, it gives
\begin{equation}\label{fully:RL}
\ba{c|cc} 0
&\frac{1}{8}+\frac{1}{4}\alpha-\frac{\sqrt{3}}{2}\beta+\frac{3}{4}\gamma&
-\frac{1}{8}+\frac{3}{4}\alpha-\frac{\sqrt{3}}{2}\beta-\frac{3}{4}\gamma\\[2pt]
\frac{2}{3}&\frac{1}{8}+\frac{1}{4}\alpha-\frac{\sqrt{3}}{6}\beta-\frac{1}{4}\gamma&
-\frac{1}{8}+\frac{3}{4}\alpha+\frac{\sqrt{3}}{2}\beta+\frac{1}{4}\gamma\\[2pt]
\hline & \frac{1}{4}& \frac{1}{4}\\[2pt]
\hline & \frac{1}{4}& \frac{3}{4} \ea
\end{equation}
\begin{equation}\label{fully:RR}
\ba{c|cc} \frac{1}{3}
&\frac{1}{8}+\frac{3}{4}\alpha-\frac{\sqrt{3}}{2}\beta+\frac{1}{4}\gamma&
-\frac{1}{8}+\frac{1}{4}\alpha+\frac{\sqrt{3}}{6}\beta-\frac{1}{4}\gamma\\[2pt]
1&\frac{1}{8}+\frac{3}{4}\alpha+\frac{\sqrt{3}}{2}\beta-\frac{3}{4}\gamma
&-\frac{1}{8}+\frac{1}{4}\alpha+\frac{\sqrt{3}}{2}\beta+\frac{3}{4}\gamma\\[2pt]
\hline & \frac{1}{2}& 0\\[2pt]
\hline & \frac{3}{4}& \frac{1}{4} \ea
\end{equation}
\begin{equation}\label{fully:LO}
\ba{c|cc} 0 &
\frac{1}{4}+\frac{\alpha}{2}-\sqrt{3}\beta+\frac{3\gamma}{2}&
-\frac{1}{4}+\frac{\alpha}{2}-\frac{3\gamma}{2}\\[2pt]
1 & \frac{1}{4}+\frac{\alpha}{2}-\frac{3\gamma}{2}&
-\frac{1}{4}+\frac{\alpha}{2}+\sqrt{3}\beta+\frac{3\gamma}{2}\\[2pt]
\hline & \frac{1}{2}& 0\\[2pt]
\hline & \frac{1}{2}& \frac{1}{2} \ea
\end{equation}
and all of which are of order 2. If we impose the diagonally
implicit requirements \eqref{diag imp:condi}, then by eliminating
$\gamma$ we then get the following two-parameter families of
diagonally implicit symplectic integrators
\begin{equation}\label{diag:RL}
\ba{c|cc} 0 &\alpha-\sqrt{3}\beta&
0\\[2pt]
\frac{2}{3}&\frac{1}{6}&
-\frac{1}{6}+\alpha+\frac{\sqrt{3}}{3}\beta\\[2pt]
\hline & \frac{1}{4}& \frac{1}{4}\\[2pt]
\hline & \frac{1}{4}& \frac{3}{4} \ea
\end{equation}
\begin{equation}\label{diag:RR}
\ba{c|cc} \frac{1}{3} &\alpha-\frac{\sqrt{3}}{3}\beta& 0\\[2pt]
1&\frac{1}{2}&-\frac{1}{2}+\alpha+\sqrt{3}\beta\\[2pt]
\hline & \frac{1}{2}& 0\\[2pt]
\hline & \frac{3}{4}& \frac{1}{4} \ea
\end{equation}
\begin{equation}\label{diag:LO}
\ba{c|cc} 0 & \alpha-\sqrt{3}\beta&0\\[2pt]
1 & \frac{1}{2}&
-\frac{1}{2}+\alpha+\sqrt{3}\beta\\[2pt]
\hline & \frac{1}{2}& 0\\[2pt]
\hline & \frac{1}{2}& \frac{1}{2} \ea
\end{equation}

We point out that, if we further require the following conditions
for explicit RKN schemes (as a very special type of diagonally
implicit schemes)
\begin{equation}\label{expl:condi}
\bar{a}_{ij}=0,\,j\geq i,
\end{equation}
then we derive a linear algebraic system in terms of
$\alpha,\,\beta,\,\gamma$ for each case, which can be easily solved
and their solutions are
\begin{itemize}
  \item[(a)]
  $\alpha=\frac{1}{8},\,\beta=\frac{\sqrt{3}}{24},\,\gamma=-\frac{1}{8}$
  for \eqref{fully:RL};
  \item[(b)] $\alpha=\frac{1}{8},\,\beta=\frac{\sqrt{3}}{8},\,\gamma=-\frac{1}{8}$
  for \eqref{fully:RR};
  \item[(c)] $\alpha=\frac{1}{4},\,\beta=\frac{\sqrt{3}}{12},\,\gamma=-\frac{1}{12}$
  for \eqref{fully:LO}.
\end{itemize}
Consequently, by substituting them into \eqref{fully:RL},
\eqref{fully:RR} and \eqref{fully:LO}, it yields the following three
explicit symplectic integrators
\[\ba{c|cc} 0 &0&0\\[2pt]\frac{2}{3}&\frac{1}{6}&
0\\[2pt]
\hline & \frac{1}{4}& \frac{1}{4}\\[2pt]
\hline & \frac{1}{4}& \frac{3}{4} \ea\qquad
\ba{c|cc} \frac{1}{3} & 0 & 0\\[2pt]
1&\frac{1}{2}&0\\[2pt]
\hline & \frac{1}{2}& 0\\[2pt]
\hline & \frac{3}{4}& \frac{1}{4} \ea \qquad
\ba{c|cc} 0 & 0& 0\\[2pt]
1 & \frac{1}{2}&0\\[2pt]
\hline & \frac{1}{2}& 0\\[2pt]
\hline & \frac{1}{2}& \frac{1}{2} \ea\] It is worth mentioning that
the right-hand tableau provides the well-known St\"{o}rmer-Verlet
scheme, which has been the most widely used scheme by far in many
fields such as astronomy, molecular dynamics and so on
\cite{hairerlw06gni}.

More higher order diagonally implicit symplectic integrators with
more stages can be constructed by the same techniques. For instance,
a 3-stage 4-order integrator can be obtained by imposing the
diagonally implicit conditions to the last table shown in Table
\ref{exa:SRKN4}, which means we should take
$\alpha=0,\,\beta=\frac{\sqrt{5}}{30}$. However, it is difficult to
construct higher order explicit symplectic integrators along the
same line, as an explicit symplectic integrators generally can be
completely determined by the nodes $c_i$ of a quadrature formula
\cite{okunbors92ecm}, which means we can not get explicit symplectic
integrators by using a given quadrature formula (e.g. the commonly
used Gaussian, Radau and Lobatto type quadrature). However, it is
possible that an explicit symplectic integrator stems from a csRKN
method by using the associated quadrature formula.

\section{Concluding remarks}

We propose the continuous-stage Runge-Kutta-Nystr\"{o}m (csRKN)
methods for solving second order ordinary differential equations in
this paper,  and the construction of symplecticity-preserving
integrators for separable Hamiltonian systems is investigated. It is
shown that the construction of csRKN methods heavily relies on the
Legendre polynomial expansion technique coupling with the symplectic
conditions and order conditions. Based on symplectic csRKN methods,
several new classes of symplectic RKN methods are obtained in use of
the quadrature formulae, and some free parameters are included in
the Butcher tableaux. It is interesting to see that we can use
different quadrature formulae to get different RKN schemes even for
the same csRKN coefficients. In addition, we can set many free
parameters to get more methods, though we only provides the methods
with two free parameters. We have only considered the methods up to
order 5 in this paper, but it is possible to construct more higher
order methods with the same technique. It is stressed that our
approach seems more easier to construct RKN type methods than the
traditional approaches which in general have to solve the tedious
nonlinear algebraic equations that stem from the order conditions
with many unknown coefficients.

\section*{Acknowledgements}

The first author was supported by the Foundation of the NNSFC
(11401055) and the Foundation of Education Department of Hunan
Province (15C0028). The second author was supported by the
foundation of NSFC(11201125),  the foundation of department of
education of Henan province(12B110010), the young growth foundation
of Henan Polytechnic university (B2390),  and State Key Laboratory
of Scientific and Engineering Computing, CAS.

\end{document}